\documentclass[a4paper]{amsart}
\usepackage{hyperref}
\usepackage{amsmath}
\usepackage{amsthm}
\usepackage[backrefs,msc-links]{amsrefs}
\usepackage{amsfonts}
\IfFileExists{cyr.fd}{\input{cyr.fd}}{}
\newtheorem{thm}{Theorem}
\newtheorem{prop}[thm]{Proposition}
\newtheorem{cor}[thm]{Corollary}

\theoremstyle{definition}
\newtheorem{defn}[thm]{Definition}
\newtheorem{example}[thm]{Example}

\theoremstyle{remark}
\newtheorem{rem}[thm]{Remark}

\providecommand{\Space}[3][]{\ensuremath{\mathbb{#2}^{#3}_{#1}{}}}
\providecommand{\algebra}[1]{\ensuremath{\mathfrak{#1}}}

\providecommand{\linv}[2][\relax]{\mathfrak{L}^{#2}_{#1}}

  \providecommand{\FSpace}[3][]{\ensuremath{\ifx#2l \ell_{#3}^{#1}{}
      \else  #2_{#3}^{#1}{}\fi}} 
\providecommand{\SL}[1][2]{\ensuremath{\FSpace{SL}{#1}(\Space{R}{})}}
\providecommand{\norm}[2][\relax]{\left\|#2\right\|\ifx#1\relax\else_{#1}\fi}
\providecommand{\modulus}[2][\relax]{\left| #2 \right|\ifx#1\relax\else_{#1}\fi}
\providecommand{\scalar}[3][\relax]{\left\langle #2,#3 
        \right\rangle\ifx#1\relax\else_{#1}\fi}
\providecommand{\eqref}[1]{\textup{(\ref{#1})}}
\providecommand{\uir}[3][0]{\ifcase #1{\rho^{#2}_{#3}}%
\or {\breve{\rho}^{#2}_{#3}}%
\or {\tilde{\rho}^{#2}_{#3}}\fi}
\providecommand{\oper}[1]{\mathcal{#1}}
\providecommand{\rmi}{\mathrm{i}}

\providecommand{\ead}[1]{\email{#1}}
\begin{document}
\title{Covariant Transform}

\author
{\href{http://amsta.leeds.ac.uk/~kisilv/}{Vladimir V. Kisil}}

\address{%
School of Mathematics, 
University of Leeds, 
Leeds LS2\,9JT, 
UK\\
(On leave from the Odessa University)
}
\ead{\href{mailto:kisilv@maths.leeds.ac.uk}{\texttt{kisilv@maths.leeds.ac.uk}}}

\maketitle

\centerline{\emph{Dedicated to the memory of Cora Sadosky}}

\begin{abstract} 
  The paper develops theory of covariant transform, which is inspired
  by the wavelet construction. It was observed that many interesting
  types of \emph{wavelets} (or \emph{coherent states}) arise from
  group representations which \emph{are not} square integrable or
  vacuum vectors which \emph{are not} admissible. Covariant transform
  extends an applicability of the popular wavelets construction to
  classic examples like the Hardy space \(\FSpace{H}{2}\), Banach
  spaces, covariant functional calculus and many others.

\end{abstract}

\tableofcontents

A general group-theoretical
construction~\cites{Perelomov86,FeichGroech89a,Kisil98a,AliAntGaz00,%
Fuhr05a,ChristensenOlafsson09a,KlaSkag85}
of \emph{wavelets} (or \emph{coherent states}) starts from an
irreducible square integrable  representation---in the
proper sense or modulo a subgroup.  Then a mother wavelet is chosen to
be admissible. This leads to a wavelet transform which is an
isometry to \(\FSpace{L}{2}\) space with respect to the Haar measure
on the group or (quasi)invariant measure on a homogeneous space.

The importance of the above situation shall not be diminished, however
an exclusive restriction to such a setup is not necessary, in fact.
Here is a classical example from complex analysis: the Hardy space
\(\FSpace{H}{2}(\Space{T}{})\) on the unit circle and Bergman spaces
\(\FSpace[n]{B}{2}(\Space{D}{})\) in the unit disk produce wavelets
associated with representations \(\rho_1\) and \(\rho_n\) of the group
\(\SL\) respectively~\cite{Kisil97c}. While representations \(\rho_n\)
are from square integrable discrete series, the mock discrete series
representation \(\rho_1\) is not square
integrable~\citelist{\cite{Lang85}*{\S~VI.5}
  \cite{MTaylor86}*{\S~8.4}}. However it would be natural to treat the
Hardy space in the same framework as Bergman ones. Some more examples
will be presented below.

\section{Covariant Transform}
\label{sec:wavelet-transform}

To make a sharp but still natural generalisation of wavelets we give the following
definition.
\begin{defn}\cite{Kisil09d}
  Let \(\uir{}{}\) be a representation of
  a group \(G\) in a space \(V\) and \(F\) be an operator from \(V\) to a space
  \(U\). We define a \emph{covariant transform}
  \(\oper{W}\) from \(V\) to the space \(\FSpace{L}{}(G,U)\) of
  \(U\)-valued functions on \(G\) by the formula:
  \begin{equation}
    \label{eq:coheret-transf-gen}
    \oper{W}: v\mapsto \hat{v}(g) = F(\uir{}{}(g^{-1}) v), \qquad
    v\in V,\ g\in G.
  \end{equation}
  Operator \(F\) will be called \emph{fiducial operator} in this context.
\end{defn}
We borrow the name for operator \(F\) from fiducial vectors of
Klauder and Skagerstam~\cite{KlaSkag85}. 
\begin{rem}
  We do not require that fiducial operator \(F\) shall be linear.
  Sometimes the homogeneity, i.e. \(F(t v)=tF(v)\) for \(t>0\), alone
  can be already sufficient, see Example~\ref{ex:maximal-function}.
\end{rem}
\begin{rem}
  \label{re:range-dim} 
  Usefulness of the covariant transform is in the reverse proportion
  to the dimensionality of the space \(U\). The covariant transform
  encodes properties of \(v\) in a function \(\oper{W}v\) on \(G\).
  For a low dimensional \(U\) this function can be ultimately
  investigated by means of harmonic analysis. Thus \(\dim U=1\)
  (scalar-valued functions) is the ideal case, however, it is
  unattainable sometimes, see Example~\ref{it:direct-funct} below. We
  may have to use a higher dimensions of \(U\) if the given group
  \(G\) is not rich enough.
\end{rem}
As we will see below covariant transform is a close relative of
wavelet transform.  The name is chosen due to the following common
property of both transformations.
\begin{thm} 
  \label{pr:inter1} 
  The covariant transform~\eqref{eq:coheret-transf-gen}
  intertwines \(\uir{}{}\) and the left regular representation
  \(\Lambda\)    on \(\FSpace{L}{}(G,U)\):
  \begin{displaymath}
    \oper{W} \uir{}{}(g) = \Lambda(g) \oper{W}.
  \end{displaymath}
  Here \(\Lambda\) is defined as usual by:
  \begin{equation}\label{eq:left-reg-repr}
    \Lambda(g): f(h) \mapsto f(g^{-1}h).
  \end{equation}
\end{thm}
\begin{proof}
  We have a calculation similar to wavelet
  transform~\cite{Kisil98a}*{Prop.~2.6}. Take \(u=\uir{}{}(g) v\) and
  calculate its covariant transform:
   \begin{eqnarray*}{}
     [\oper{W}( \uir{}{}(g) v)] (h) & = &  [\oper{W}(\uir{}{}(g) v)] (h)=F(\uir{}{}(h^{-1}) \uir{}{}(g) v ) \\
     & = & F(\uir{}{}((g^{-1}h)^{-1}) v) \\
     & = & [\oper{W}v] (g^{-1}h)\\
     & = & \Lambda(g) [\oper{W}v] (h).
   \end{eqnarray*}
\end{proof}
The next result follows immediately:
\begin{cor}\label{co:pi}
  The image space \(\oper{W}(V)\) is invariant under the
  left shifts on \(G\).
\end{cor}


\section{Examples of Covariant Transform}
\label{sec:exampl-covar-transf}
We start from the classical example of the group-theoretical wavelet transform:
\begin{example}
  \label{ex:wavelet}
  Let \(V\) be a Hilbert space with an inner product
  \(\scalar{\cdot}{\cdot}\) and \(\uir{}{}\) be a unitary
  representation of a group \(G\) in the space \(V\). Let \(F: V
  \rightarrow \Space{C}{}\) be a functional \(v\mapsto
  \scalar{v}{v_0}\) defined by a vector \(v_0\in V\). The vector
  \(v_0\) is oftenly called the \emph{mother wavelet} in areas related
  to signal processing or the \emph{vacuum state} in quantum
  framework.

  Then the transformation~\eqref{eq:coheret-transf-gen} is the
  well-known expression for a \emph{wavelet
    transform}~\cite[(7.48)]{AliAntGaz00} (or \emph{representation
    coefficients}):
  \begin{equation}
    \label{eq:wavelet-transf}
    \oper{W}: v\mapsto \hat{v}(g) = \scalar{\uir{}{}(g^{-1})v}{v_0}  =
    \scalar{ v}{\uir{}{}(g)v_0}, \qquad
    v\in V,\ g\in G.
  \end{equation}
  The family of vectors \(v_g=\uir{}{}(g)v_0\) is called
  \emph{wavelets} or \emph{coherent states}. In this case we obtain
  scalar valued functions on \(G\), thus the fundamental r\^ole of
  this example is explained in Rem.~\ref{re:range-dim}.

  This scheme is typically carried out for a square integrable
  representation \(\uir{}{}\) and \(v_0\) being an admissible
  vector~\cites{Perelomov86,FeichGroech89a,AliAntGaz00,Fuhr05a,ChristensenOlafsson09a}.
  In this case the wavelet (covariant) transform is a map into the
  square integrable functions~\cite{DufloMoore} with respect to the
  left Haar measure. The map becomes an isometry if \(v_0\) is properly
  scaled.
\end{example}
However square integrable representations and admissible vectors does not cover all
interesting cases.
\begin{example}
  \label{ex:ax+b}
  Let \(G=\mathrm{Aff}\) be the ``\(ax+b\)'' (or \emph{affine})
  group~\cite[\S~8.2]{AliAntGaz00}: the set of points \((a,b)\),
  \(a\in \Space[+]{R}{}\), \(b\in \Space{R}{}\) in the upper
  half-plane with the group law:
   \begin{equation}
     (a, b) * (a', b') = (aa', ab'+b)
   \end{equation}
   and left invariant measure \(a^{-2}\,da\,db\).  Its isometric
   representation on \(V=\FSpace{L}{p}(\Space{R}{})\) is given by the
   formula:
  \begin{equation}\label{eq:ax+b-repr-1}
     [\uir{}{p}(g)\, f](x)= a^{\frac{1}{p}}f\left(ax+b\right),
     \qquad\text{where } g^{-1}=(a,b).
  \end{equation}
  We consider the operators \(F_{\pm}:\FSpace{L}{2}(\Space{R}{})
  \rightarrow \Space{C}{}\) defined by:
  \begin{equation}
    \label{eq:cauchy}
    F_{\pm}(f)=\frac{1}{2\pi i}\int_{\Space{R}{}} \frac{f(t)\,dt}{x\mp \rmi}.
  \end{equation}
  Then the covariant transform~\eqref{eq:coheret-transf-gen} is the
  Cauchy integral from \(\FSpace{L}{p}(\Space{R}{})\) to the space of
  functions \(\hat{f}(a,b)\) such that
  \(a^{-\frac{1}{p}}\hat{f}(a,b)\) is in the Hardy space in the
  upper/lower half-plane \(\FSpace{H}{p}(\Space[\pm]{R}{2})\).
  Although the representation~\eqref{eq:ax+b-repr-1} is square
  integrable for \(p=2\), the function \(\frac{1}{x\pm \rmi}\) used
  in~\eqref{eq:cauchy} is not an admissible vacuum vector.  Thus the
  complex analysis become decoupled from the traditional wavelet
  theory. As a result the application of wavelet theory shall relay on
  an extraneous mother wavelets~\cite{Hutnik09a}.

  Many important objects in complex analysis are generated by
  inadmissible mother wavelets like~\eqref{eq:cauchy}. For example, if
  \(F:\FSpace{L}{2}(\Space{R}{}) \rightarrow \Space{C}{}\) is defined
  by \(F: f \mapsto F_+ f + F_-f\) then the covariant
  transform~\eqref{eq:coheret-transf-gen} reduces to the \emph{Poisson
    integral}.  If \(F:\FSpace{L}{2}(\Space{R}{}) \rightarrow
  \Space{C}{2}\) is defined by \(F: f \mapsto( F_+ f, F_-f)\) then the
  covariant transform~\eqref{eq:coheret-transf-gen} represents a
  function \(f\) on the real line as a jump:
  \begin{equation}
    \label{eq:jump-rl}
    f(z)=f_+(z)-f_-(z),\qquad f_\pm(z)\in \FSpace{H}{p}(\Space[\pm]{R}{2})
  \end{equation}
  between functions analytic in the upper and the lower half-planes.
  This makes a decomposition of \(\FSpace{L}{2}(\Space{R}{})\) into
  irreducible components of the representation~\eqref{eq:ax+b-repr-1}.
  Another interesting but non-admissible vector is the Gaussian
  \(e^{-x^2}\).
\end{example}
\begin{example}
  \label{ex:sl2}
  For the group \(G=\SL\)~\cite{Lang85} let us consider the unitary
  representation \(\uir{}{}\) on the space of square integrable function
  \(\FSpace{L}{2}(\Space[+]{R}{2})\) on the upper half-plane through
  the M\"obius transformations:
   \begin{equation}
     \label{eq:sl2-action}
     \uir{}{}(g): f(z) \mapsto \frac{1}{(c z + d)^2}\,
     f\left(\frac{a z+ b }{c z +d}\right), \qquad g^{-1}=\
     \begin{pmatrix}
       a & b \\ c & d 
     \end{pmatrix}.
   \end{equation}
   This is a representation from the discrete series and
   \(\FSpace{L}{2}(\Space{D}{})\) and irreducible invariant subspaces
   are parametrised by integers.  Let \(F_k\) be the functional
   \(\FSpace{L}{2}(\Space[+]{R}{2})\rightarrow \Space{C}{}\) of
   pairing with the lowest/highest \(k\)-weight vector in the
   corresponding irreducible component
   \(\FSpace{B}{k}(\Space[\pm]{R}{2})\), \(k\geq 2\) of the discrete
   series~\cite[Ch.~VI]{Lang85}. Then we can build an operator \(F\)
   from various \(F_k\) similarly to the previous Example. In
   particular, the jump representation~\eqref{eq:jump-rl} on the real
   line generalises to the representation of a square integrable
   function \(f\) on the upper half-plane as a sum
   \begin{displaymath}
     f(z)=\sum_k a_k f_k(z), \qquad f_k\in\FSpace{B}{n}(\Space[\pm]{R}{2})
   \end{displaymath}
   for prescribed coefficients \(a_k\) and analytic functions \(f_k\) in
   question from different irreducible subspaces.

   Covariant transform is also meaningful for principal and
   complementary series of representations of the group
   \(\SL\), which are not square integrable~\cite{Kisil97c}.
\end{example}
\begin{example}
  \label{it:direct-funct}
  A straightforward generalisation of Ex.~\ref{ex:wavelet} is
  obtained if \(V\) is a Banach space and \(F: V \rightarrow
  \Space{C}{}\) is an element of \(V^*\). Then the
  covariant transform coincides with the construction of wavelets in
  Banach spaces~\cite{Kisil98a}.
\end{example}
\begin{example}
  The next stage of generalisation is achieved if \(V\) is a
  Banach space and \(F: V \rightarrow \Space{C}{n}\) is a linear
  operator. Then the corresponding covariant transform is a map
  \(\oper{W}: V \rightarrow \FSpace{L}{}(G,\Space{C}{n})\). This is
  closely related to M.G.~Krein's works on \emph{directing
    functionals}~\cite{Krein48a}, see also \emph{multiresolution
    wavelet analysis}~\cite{BratJorg97a}, Clifford-valued
  Bargmann spaces~\cite{CnopsKisil97a} and~\cite[Thm.~7.3.1]{AliAntGaz00}.
\end{example}
\begin{example}
  Let \(F\) be a projector \(\FSpace{L}{p}(\Space{R}{})\rightarrow
  \FSpace{L}{p}(\Space{R}{})\) defined by the relation \((Ff)\hat{\
  }(\lambda )=\chi(\lambda)\hat{f}(\lambda)\), where the hat denotes the Fourier
  transform and \(\chi(\lambda)\) is the characteristic function of
  the set \([-2,-1]\cup[1,2]\). Then the covariant transform
  \(\FSpace{L}{p}(\Space{R}{})\rightarrow \FSpace{C}{}(\mathrm{Aff},
  \FSpace{L}{p}(\Space{R}{}))\) generated by the
  representation~\eqref{eq:ax+b-repr-1} of the affine group from
  \(F\) contains all information provided by the \emph{Littlewood--Paley
  operator}~\cite{Grafakos08}*{\S~5.1.1}.
\end{example}
\begin{example}
  \label{ex:maximal-function}
  A step in a different direction is a consideration of
  non-linear operators. Take again the ``\(ax+b\)'' group and its
  representation~\eqref{eq:ax+b-repr-1}. 
  We define \(F\) to be a homogeneous but non-linear functional
  \(V\rightarrow \Space[+]{R}{}\):
  \begin{displaymath}
    F (f) = \frac{1}{2}\int\limits_{-1}^1 \modulus{f(x)}\,dx.
  \end{displaymath}
  The covariant transform~\eqref{eq:coheret-transf-gen} becomes:
  \begin{equation}
    \label{eq:hardi-max}
    [\oper{W}_p f](a,b) =  F(\uir{}{p}(a,b) f) 
    = \frac{1}{2}\int\limits_{-1}^1
    \modulus{a^{\frac{1}{p}}f\left(ax+b\right)}\,dx
    = a^{\frac{1}{p}}\frac{1}{2a}\int\limits^{b+a}_{b-a}
    \modulus{f\left(x\right)}\,dx.
  \end{equation}
  Obviously \(M_f(b)=\max_{a}[\oper{W}_{\infty}f](a,b)\) coincides
  with the Hardy \emph{maximal function}, which contains important
  information on the original function \(f\). 
  From the Cor.~\ref{co:pi} we deduce that the operator \(M: f\mapsto
  M_f\) intertwines \(\uir{}{p}\) with itself \(\uir{}{p}M=M
  \uir{}{p}\).

  Of course, the full covariant transform~\eqref{eq:hardi-max} is even more
  detailed than \(M\). For example,
  \(\norm{f}=\max_b[\oper{W}_{\infty}f](\frac{1}{2},b)\) is the shift
  invariant norm~\cite{Johansson08a}.
\end{example}
\begin{example}
  Let \(V=\FSpace{L}{c}(\Space{R}{2})\) be the space of
  compactly supported bounded functions on the plane. We take \(F\)
  be the linear operator \(V\rightarrow \Space{C}{}\) of integration
  over the real line:
  \begin{displaymath}
    F: f(x,y)\mapsto F(f)=\int_{\Space{R}{}}f(x,0)\,dx.
  \end{displaymath}
  Let \(G\) be the group of Euclidean motions of the plane
  represented by \(\uir{}{}\) on \(V\) by a change of variables. Then
  the wavelet transform \(F(\uir{}{}(g)f)\) is the \emph{Radon
    transform}.
\end{example}
\begin{example}
  Let a representation \(\uir{}{}\) of a group \(G\) act on a
  space \(X\). Then there is an associated representation
  \(\uir{}{B}\) of \(G\) on a space \(V=\FSpace{B}{}(X,Y)\) of
  linear operators \(X\rightarrow Y\) defined by the identity:
  \begin{equation}
    \label{eq:oper-repres}
    (\uir{}{B}(g) A)x=A(\uir{}{}(g^{-1})x), \qquad x\in X,\ g\in G,\ A
    \in \FSpace{B}{}(X,Y). 
  \end{equation}
  Following the Remark~\ref{re:range-dim} we take \(F\) to be a
  functional \(V\rightarrow\Space{C}{}\), for example \(F\) can be
  defined from a pair \(x\in X\), \(l\in Y^*\)  by the expression
  \(F: A\mapsto \scalar{Ax}{l}\). Then the covariant
  transform is:
  \begin{displaymath}
    \oper{W}: A \mapsto \hat{A}(g)=F(\uir{}{B}(g) A).
  \end{displaymath}
  This is an example of \emph{covariant calculus}~\cite{Kisil98a,Kisil02a}.
\end{example}
\begin{example}
  A modification of the previous construction is obtained if we
  have two groups \(G_1\) and \(G_2\) represented by \(\uir{}{1}\)
  and \(\uir{}{2}\) on \(X\) and \(Y^*\) respectively. Then we have a covariant
  transform \(\FSpace{B}{}(X,Y)\rightarrow \FSpace{L}{}(G_1\times
  G_2, \Space{C}{})\) defined by the formula:
  \begin{displaymath}
    \oper{W}: A \mapsto \hat{A}(g_1,g_2)=\scalar{A\uir{}{1}(g_1)x}{\uir{}{2}(g_2)l}.
  \end{displaymath}
  This generalises \emph{Berezin functional calculi}~\cite{Kisil98a}.
\end{example}
\begin{example}
  Let us restrict the previous example to the case when \(X=Y\) is a
  Hilbert space, \(\uir{}{1}{}=\uir{}{2}{}=\uir{}{}\) and \(x=l\) with
  \(\norm{x}=1\). Than the range of the covariant transform:
  \begin{displaymath}
    \oper{W}: A \mapsto \hat{A}(g)=\scalar{A\uir{}{}(g)x}{\uir{}{}(g)x}
  \end{displaymath}
  is a subset of the \emph{numerical range} of the operator \(A\). As
  a function on a group \(\hat{A}(g)\) provides a better description of
  \(A\) than the set of its values---numerical range. 
\end{example}
\begin{example}
  The group \(SU(1,1)\simeq \SL\) consists of \(2\times 2\) matrices
  of the form \(\begin{pmatrix}
    \alpha&\beta\\\bar{\beta}&\bar{\alpha}
  \end{pmatrix}\) with the unit
  determinant~\cite[\S~IX.1]{Lang85}. Let \(T\) be an operator
  with the spectral radius less than \(1\). Then the associated
  M\"obius transformation
  \begin{equation}
    \label{eq:meobius-T}
    g: T \mapsto g\cdot T =  \frac{\alpha T+\beta
      I}{\bar{\beta}T+\bar{\alpha}I}, \qquad \text{where} \quad
    g=
    \begin{pmatrix}
      \alpha&\beta\\\bar{\beta}&\bar{\alpha}
    \end{pmatrix}\in \SL,\ 
  \end{equation}
  produces a well-defined operator with the spectral radius less than
  \(1\) as well.  Thus we have a representation of \(SU(1,1)\). A choice
  of an operator \(F\) will define the corresponding covariant
  transform. In this way we obtain generalisations of
  \emph{Riesz--Dunford functional calculus}~\cite{Kisil02a}.
\end{example}

\begin{example}
  Consider again the action~\eqref{eq:meobius-T} of the Moebius
  transformations on operators from the previous Example. Let us
  introduce the defect operators \(D_T=(I-T^*T)^{1/2}\) and
  \(D_{T^*}=(I-TT^*)^{1/2}\). For the case \(F=D_{T^*}\) the covariant
  transform is, cf.~\cite{NagyFoias70}*{\S~VI.1, (1.2)}:
  \begin{displaymath}{}
    [\oper{W} T](g)=F(g\cdot T)=-e^{\rmi\phi}\,\Theta_T(z)\, D_T, \qquad
    \text{for } 
    g
    =     \begin{pmatrix}
      e^{\rmi\phi/2}&0\\0&e^{-\rmi\phi/2}
    \end{pmatrix}
    \begin{pmatrix}
      1&-z\\-\bar{z}&1
    \end{pmatrix},
  \end{displaymath}
  where 
  the  \emph{characteristic function}
  \(\Theta_T(z)\)~\cite{NagyFoias70}*{\S~VI.1, (1.1)} is:
  \begin{displaymath}
    \Theta_T(z) = -T+D_{T^*}\,(I-zT^*)^{-1}\,z\,D_T.
  \end{displaymath}
  Thus we approached the \emph{functional model} of operators from the
  covariant transform. In accordance with Remark~\ref{re:range-dim}
  the model is most fruitful for the case of operator
  \(F=D_{T^*}\) being one-dimensional.
\end{example}

\section{Induced Covariant Transform}
\label{sec:cauchy-transform}
The choice of fiducial operator \(F\) can significantly influence the
behaviour of the covariant transform.  Let \(G\) be a group and
\({H}\) be its closed subgroup with the corresponding homogeneous
space \(X=G/{H}\). Let \(\uir{}{}\) be a representation of \(G\) by
operators on a space \(V\), we denote by \(\uir{}{H}\) the restriction
of \(\uir{}{}\) to the subgroup \(H\). 
\begin{defn}
  Let \(\chi\) be a representation of the subgroup \({H}\) in a space \(U\) and
  \(F: V\rightarrow U\) be an intertwining operator between \(\chi\)
  and the representation \(\uir{}{H}\):
  \begin{displaymath}
   F(\uir{}{}(h) v)=F(v)\chi(h), \qquad \text{ for all }h\in {H},\
   v\in V.
  \end{displaymath}
  Then the covariant transform~\eqref{eq:coheret-transf-gen} generated
  by \(F\) is called the \emph{induced covariant transform}.
\end{defn}
The following is the main motivating example.
\begin{example}
  Consider the traditional wavelet transform as outlined in
  Ex.~\ref{ex:wavelet}. Chose a vacuum vector \(v_0\) to be a joint
  eigenvector for all operators \(\uir{}{}(h)\), \(h\in H\), that is
  \(\uir{}{}(h) v_0=\chi(h) v_0\), where \(\chi(h)\) is a complex
  number depending of \(h\). Then \(\chi\) is obviously a character of
  \(H\). 

  The image of wavelet transform~\eqref{eq:wavelet-transf} with such a
  mother wavelet will have a property:
  \begin{displaymath}
    \hat{v}(gh) = \scalar{ v}{\uir{}{}(gh)v_0} 
    = \scalar{v}{\uir{}{}(g)\chi(h)v_0}
    =\chi(h)\hat{v}(g).
  \end{displaymath}
  Thus the wavelet transform is uniquely defined by cosets on the
  homogeneous space \(G/H\).  In this case we can speak about the
  \emph{reduced wavelet transform}~\cite{Kisil97a}. 
   A representation \(\uir{}{0}\) is \emph{square integrable} \(\mod H\)
   if the induced wavelet transform \([\oper{W}f_0](w)\) of the vacuum
   vector \(f_0(x)\) is square integrable on \(X\).
\end{example}
The image of induced covariant transform have the similar property:
\begin{equation}
  \label{eq:induced-covariant}
  \hat{v}(gh)=F(\uir{}{}((gh)^{-1})
  v)=F(\uir{}{}(h^{-1})\uir{}{}(g^{-1}) v)
  =F(\uir{}{}(g^{-1}) v)\chi{}{}(h^{-1}).
\end{equation}
Thus it is enough to know the value of the covariant transform only at a
single element in every coset \(G/H\) in order to reconstruct it for
the entire group \(G\) by the representation \(\chi\). Since coherent
states (wavelets) are now parametrised by points
homogeneous space \(G/H\) they are referred sometimes as coherent
states which are not connected to a group~\cite{Klauder96a}, however
this is true only in a very narrow sense as explained above.
\begin{example}
  To make it more specific we can consider the representation of \(\SL\)
  defined on \(\FSpace{L}{2}(\Space{R}{})\) by the formula,
  cf.~\eqref{eq:sl2-action}:
  \begin{displaymath}
    \uir{}{}(g): f(z) \mapsto \frac{1}{(c x + d)}\,
    f\left(\frac{a x+ b }{c x +d}\right), \qquad g^{-1}=\
    \begin{pmatrix}
      a & b \\ c & d 
    \end{pmatrix}.
  \end{displaymath}
  Let \(K\subset\SL\) be the compact subgroup 
  of matrices \(
  h_t=
  \begin{pmatrix}
    \cos t&\sin t\\-\sin t&\cos t
  \end{pmatrix}\). Then for the fiducial operator
  \(F_{\pm}\)~\eqref{eq:cauchy} we have
  \(F_{\pm}\circ\uir{}{}(h_t)=e^{\mp\rmi t}F_{\pm}\). Thus we can
  consider the covariant transform only for points in \(\SL/K\),
  however this set can be naturally identified with the \(ax+b\)
  group. Thus we do not obtain any advantage of extending the group in
  Ex.~\ref{ex:ax+b} from \(ax+b\) to \(\SL\) if we will be still using the
  fiducial operator \(F_\pm\)~\eqref{eq:cauchy}.
\end{example}
Functions on the group \(G\), which have the property
\(\hat{v}(gh)=\hat{v}(g)\chi(h)\)~\eqref{eq:induced-covariant},
provide a space for the representation of \(G\) induced by the
representation \(\chi\) of the subgroup \(H\). This explains the
choice of the name for induced covariant transform.

\begin{rem}
  Induced covariant transform uses the fiducial operator \(F\) which
  passes through the action of the subgroup \({H}\). This reduces
  information which we obtained from this transform in some cases.
\end{rem}

There is also a simple connection between a covariant transform and
right shifts: 
\begin{prop}
  Let \(G\) be a Lie group and \(\uir{}{}\) be a representation of
  \(G\) in a space \(V\). Let \([\oper{W}f](g)=F(\uir{}{}(g^{-1})f)\) be a
  covariant transform defined by the fiducial operator \(F: V \rightarrow U\).
  Then the right shift \([\oper{W}f](gg')\) by \(g'\) is the covariant transform
  \([\oper{W'}f](g)=F'(\uir{}{}(g^{-1})f)]\) defined by the fiducial operator
  \(F'=F\circ\uir{}{}(g^{-1})\). 

  In other words the covariant transform intertwines right shifts with
  the associated action \(\uir{}{B}\)~\eqref{eq:oper-repres} on
  fiducial operators.
\end{prop}
Although the above result is obvious, its infinitesimal version has
interesting consequences.
\begin{cor}
  \label{co:cuachy-riemann}
  Let \(G\) be a Lie group with a Lie algebra \(\algebra{g}\) and
  \(\uir{}{}\) be a smooth representation of \(G\). We denote by
  \(d\uir{}{B}\) the derived representation of the associated
  representation \(\uir{}{B}\)~\eqref{eq:oper-repres} on fiducial
  operators.

  Let a fiducial operator \(F\) be a null-solution, i.e. \(A F=0\),
  for the operator \(A=\sum_J a_j d\uir{X_j}{B}\), where
  \(X_j\in\algebra{g}\) and \(a_j\) are constants.  Then the wavelet
  transform \([\oper{W} f](g)=F(\uir{}{}(g^{-1})f)\) for any \(f\)
  satisfies:
  \begin{displaymath}
    D F(g)= 0, \qquad \text{where} \quad
    D=\sum_j \bar{a}_j\linv{X_j}.
  \end{displaymath}
  Here \(\linv{X_j}\) are the left invariant fields (Lie derivatives) on
  \(G\) corresponding to \(X_j\).
\end{cor}
\begin{example}
  Consider the representation \(\uir{}{}\)~\eqref{eq:ax+b-repr-1} of the \(ax+b\)
  group with the \(p=1\). Let \(A\) and \(N\) be the basis of the
  corresponding Lie algebra generating one-parameter subgroups
  \((e^t,0)\) and \((0,t)\). Then the derived representations are:
  \begin{displaymath}
    [d\uir{A}{} f](x)= f(x)+xf'(x), \qquad [d\uir{N}{}f](x)=f'(x).
  \end{displaymath}
  The corresponding left invariant vector fields on \(ax+b\) group
  are:
  \begin{displaymath}
   \linv{A} =a \partial_a,\qquad \linv{N}=a\partial_b.
  \end{displaymath}
  The mother wavelet \(\frac{1}{x+\rmi}\) is a null solution of the
  operator \(d\uir{A}{} +\rmi d\uir{N}{}=I+(x+\rmi)\frac{d}{dx}\).
  Therefore the covariant transform with the fiducial operator
  \(F_+\)~\eqref{eq:cauchy} will consist with the null solutions to
  the operator \(\linv{A}-\rmi\linv{N}=-\rmi a(\partial_b+\rmi\partial_a)\),
  that is in the essence the Cauchy-Riemann operator in the upper
  half-plane. 
\end{example}
There is a statement which extends the previous Corollary from
differential operators to integro-differential ones. We will formulate
it for the wavelets setting.
\begin{cor}
  \label{co:cuachy-riemann-integ}
  Let \(G\) be a Lie group with a Lie algebra \(\algebra{g}\) and
  \(\uir{}{}\) be a unitary representation of \(G\), which can be
  extended to a vector space \(V\) of functions or distributions on \(G\).
  
  Let a mother wavelet \(w\in V'\)  satisfy the equation 
  \begin{displaymath}
    \int_{G} a(g)\, \uir{}{}(g) w\,dg=0,
  \end{displaymath}
  for a fixed distribution \(a(g) \in V\). Then  any wavelet transform
  \(F(g)= \oper{W} f(g)=\scalar{f}{\uir{}{}(g)w_0}\) obeys the condition:
  \begin{displaymath}
   DF=0,\qquad \text{where} \quad D=\int_{G} \bar{a}(g)\, R(g) \,dg,
  \end{displaymath}
  with \(R\) being the right regular representation of \(G\).
\end{cor}
Clearly the Corollary~\ref{co:cuachy-riemann} is a particular case of
Corollary~\ref{co:cuachy-riemann-integ}.


\section{Inverse Covariant Transform}
\label{sec:invar-funct-groups}
An object invariant under the left action
\(\Lambda\)~\eqref{eq:left-reg-repr} is called \emph{left invariant}.
For example, 
let \(L\) and \(L'\) be two left invariant spaces of functions on
\(G\).  We say that a pairing \(\scalar{\cdot}{\cdot}: L\times
L' \rightarrow \Space{C}{}\) is \emph{left invariant} if
\begin{equation}
  \scalar{\Lambda(g)f}{\Lambda(g) f'}= \scalar{f}{f'}, \quad \textrm{ for all }
  \quad f\in L,\  f'\in L'.
\end{equation}
\begin{rem}
  \begin{enumerate}
  \item We do not require the pairing to be linear in general.
  \item If the pairing is invariant on space \(L\times L'\) it is not
    necessarily invariant (or even defined) on the whole
    \(\FSpace{C}{}(G)\times \FSpace{C}{}(G)\).
  \item In a more general setting we shall study an invariant pairing
    on a homogeneous spaces instead of the group. However due to length
    constraints we cannot consider it here beyond the Example~\ref{ex:hs-pairing}.
  \item An invariant pairing on \(G\) can be obtained from an invariant
    functional \(l\) by the formula \(\scalar{f_1}{f_2}=l(f_1\bar{f}_2)\).
  \end{enumerate}
\end{rem}

For a representation \(\uir{}{}\) of \(G\) in \(V\) and \(v_0\in V\)
we fix a function \(w(g)=\uir{}{}(g)v_0\). We assume that the pairing
can be extended in its second component to this \(V\)-valued
functions, say, in the weak sense.
\begin{defn}
  \label{de:admissible}
  Let \(\scalar{\cdot}{\cdot}\) be a left invariant pairing on
  \(L\times L'\) as above, let \(\uir{}{}\) be a representation of
  \(G\) in a space \(V\), we define the function
  \(w(g)=\uir{}{}(g)v_0\) for \(v_0\in V\). The \emph{inverse
    covariant transform} \(\oper{M}\) is a map \(L \rightarrow V\)
  defined by the pairing:
  \begin{equation}
    \label{eq:inv-cov-trans}
    \oper{M}: f \mapsto \scalar{f}{w}, \qquad \text{
      where } f\in L. 
  \end{equation}
\end{defn}

 \begin{example}
   Let \(G\) be a group with a unitary square integrable representation \(\rho\).
   An invariant pairing of two square integrable functions is obviously done by the
   integration over the Haar measure:
   \begin{displaymath}
     \scalar{f_1}{f_2}=\int_G f_1(g)\bar{f}_2(g)\,dg.
   \end{displaymath}
   
   For an admissible vector \(v_0\)~\cite{DufloMoore},
   \cite[Chap.~8]{AliAntGaz00} the inverse covariant transform is
   known in this setup as a \emph{reconstruction formula}.
 \end{example}
 \begin{example}
   \label{ex:hs-pairing}
   Let \(\rho\) be a square integrable representation of \(G\) modulo a subgroup
   \(H\subset G\) and let \(X=G/H\) be the corresponding homogeneous
   space with a quasi-invariant measure \(dx\).  Then integration over
   \(dx\) with an appropriate weight produces an invariant pairing.
   The inverse covariant transform is a more general
   version~\cite[(7.52)]{AliAntGaz00} of the \emph{reconstruction
     formula} mentioned in the previous example.
 \end{example}
 

 Let \(\rho\) be not a square integrable representation (even modulo a subgroup) or
 let \(v_0\) be inadmissible vector of a square integrable  representation
 \(\rho\). An invariant pairing in this case is not associated with an
 integration over any non singular invariant measure on \(G\). In this
 case we have a \emph{Hardy pairing}. The following example explains
 the name.
\begin{example}
  Let \(G\) be the ``\(ax+b\)'' group and its representation
  \(\uir{}{}\)~\eqref{eq:ax+b-repr-1} from Ex.~\ref{ex:ax+b}. An
  invariant pairing on \(G\), which is not generated by the Haar
  measure \(a^{-2}da\,db\), is:
  \begin{equation}
    \label{eq:hardy-pairing}
    \scalar{f_1}{f_2}=
    \lim_{a\rightarrow 0}\int\limits_{-\infty}^{\infty}
    f_1(a,b)\,\bar{f}_2(a,b)\,db.
  \end{equation}
  For this pairing we can consider functions \(\frac{1}{2\pi i
    (x+i)}\) or \(e^{-x^2}\), which are not admissible vectors in the
  sense of square integrable representations. Then the inverse covariant transform
  provides an \emph{integral resolutions} of the identity.
\end{example}
Similar pairings can be defined for other semi-direct products of two
groups. We can also extend a Hardy pairing to a group, which has a
subgroup with such a pairing.
\begin{example}
  Let \(G\) be the group \(\SL\) from the Ex.~\ref{ex:sl2}. Then
  the ``\(ax+b\)'' group is a subgroup of \(\SL\), moreover we can
  parametrise \(\SL\) by triples \((a,b,\theta)\),
  \(\theta\in(-\pi,\pi]\) with the respective Haar
  measure~\cite[III.1(3)]{Lang85}. Then the Hardy
  pairing
  \begin{equation}
    \label{eq:hardy-pairing1}
    \scalar{f_1}{f_2}= \lim_{a\rightarrow 0}\int\limits_{-\infty}^{\infty}
    f_1(a,b,\theta)\,\bar{f}_2(a,b,\theta)\,db\,d\theta.
  \end{equation}
  is invariant on \(\SL\) as well.  The corresponding inverse
  covariant transform provides even a finer resolution of the identity
  which is invariant under conformal mappings of the Lobachevsky
  half-plane.
\end{example}

A further study of covariant transform shall be
continued elsewhere.

\textbf{Acknowledgement.} Author is grateful to the anonymous referee
for many helpful suggestions.

\small
\bibliographystyle{plain}
\bibliography{abbrevmr,akisil,analyse,arare,aphysics}
\end{document}